\documentclass{article}
\usepackage{graphicx} 
\usepackage{hyperref}   
\hypersetup{
	colorlinks=true,
	linkcolor=blue,
	citecolor=blue,
	urlcolor=blue,
}
\usepackage[square,numbers]{natbib}
\usepackage{xcolor}
\usepackage{amsthm}
\usepackage{algorithm}
\usepackage{algpseudocode}
\usepackage{nicematrix}
\usepackage{graphicx}
\usepackage{amsfonts}
\usepackage{amsmath,mathtools,amssymb,amsfonts}

\newcommand{\norm}[1]{\left\lVert#1\right\rVert}
\newcommand{\xddots}{%
  \raise 4pt \hbox {.}
  \mkern 6mu
  \raise 1pt \hbox {.}
  \mkern 6mu
  \raise -2pt \hbox {.}
}
\DeclareMathOperator*{\SubjectTo}{Subject\phantom{a}to:}
\DeclareMathOperator*{\Minimize}{Minimize:}

\DeclareMathOperator*{\argmin}{arg\,min}

\newtheorem{theorem}{\bf{Theorem}}

\newtheorem{proposition}{\bf{Proposition}}
\newtheorem{remark}{\bf{Remark}}

\title{Predictive Energy Management for Hybrid Powertrains}
\author{Satish Vedula, and Olugbenga Moses Anubi}

\date{}

\begin{document}

\maketitle
\begin{centering}
Department of Electrical and Computer Engineering, the Center for Advanced Power Systems, Florida State University
E-mail: \{svedula, oanubi\}@fsu.edu
\end{centering}

\textbf{COPYRIGHTED. Submitted to IFAC MECC 2026 and ASME Letters in Dynamic Systems and Control Joint Submission (Under Review)}
\section{Abstract}
Hybrid power trains (HPT) run on multiple energy sources, often involving energy storage systems/batteries (ESS). As a result, the risk of battery degradation and the reliability of energy storage elements pose a major challenge in designing an energy-efficient hybrid power train. This paper presents an energy management strategy that adaptively splits power demand between the engine and the battery pack in a hybrid power train taking into account the battery degradation. Incorporating the battery degradation model directly into the underlying optimization problem is challenging on multiple fronts: 1) Any reasonable degradation model will, due to its complexity, result in a complicated optimization problem that is impractical for real-time implementation 2) the models contain a lot of time-varying parameters that can only be determined through destructive experimental procedures. As a result, it is essential to devise heuristics that reasonably capture the degradation per usage of the batteries. One such heuristic considered in this paper is the absolute power extracted from the battery. A distributed model predictive strategy is then developed to coordinate the power split to maximize efficiency while mitigating the failure risk due to battery degradation. The designed EM strategy is demonstrated through a realistic simulation of three different hybrid power trains: hybrid road vehicles (for example: a hybrid electric vehicle (HEV)), hybrid surface vehicles (for example: dynamically positioned hybrid ships (DPS)), and hybrid aerial vehicles (for example: hybrid electric aircraft (HEA)). The results show the effectiveness of the energy management strategy in managing battery degradation.

\section{Introduction}

An HPT comprises multiple forms of energy-generating sources namely: engine and battery. Over the years hybridization of traditional power trains has broadened across all manufacturing industries for example automobile industry (hybrid electric vehicles), the naval and shipping industry (hybrid electric ships), and the aviation industry (hybrid electric aircraft). As mentioned, one such example of an HPT is a hybrid electric vehicle (HEV). Stringent regulations on environmental emissions and the impending extreme limitations on the consumption of oil and gas products are prompting the automotive industry to shift focus to more alternate-energy-dependent vehicles such as fuel cell hybrid vehicles (FHVs), battery Electric vehicles (BEVs) and hybrid electric vehicles (HEVs) \cite{CCChan}. Optimal energy management (EM) between engine and battery plays a pivotal role in enhancing the fuel efficiency of HEV, but achieving it is extremely challenging due to the changing modes of operation of HEV. Rule-based Control (RBC) is one of the HEV EM methods that gained traction during the inception of HEVs \cite{banvait}. However, it is often difficult to encapsulate and capture all the system attributes into RBC to facilitate seamless operation. As a result, enhancements in computational devices have led to a substantial increase in the use of optimization-based methods such as model predictive control (MPC) \cite{Borhan}\cite{east}, reinforcement learning (RL) \cite{HuX}\cite{LinX}. 

Improvement of fuel economy and range extension are two of the major key performance indicators in the EM of HEVs, with cost functions defined in the objective of MPC problem to consider fuel injection and efficient engine operation. Numerous investigators have proposed algorithms that made fuel minimization a key part of the optimization problem \cite{Borhan,east,vaya,Moura,Golebiewski,east3,shen}. However, it is a well-known fact that the batteries degrade faster than the engines. Thus, it is important to prescribe optimization problems that consider fuel consumption while simultaneously managing battery degradation. 

\begin{figure}[t!]
\centerline{\includegraphics[width=0.95\columnwidth]{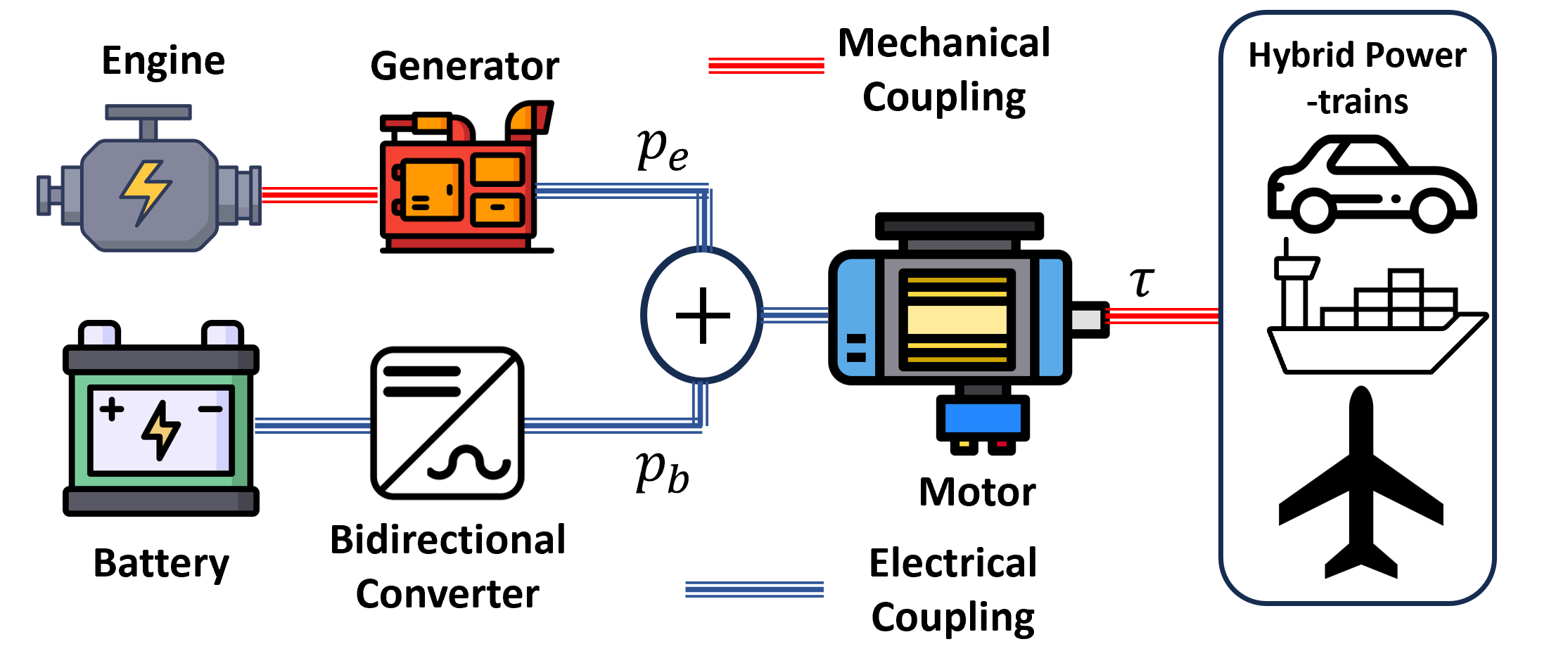}}
\caption{Hybrid power-train schematic with different applications to road vehicles (hybrid electric cars), surface vehicles (dynamically positioned ships), and aerial vehicles (hybrid electric aircraft).}
\label{HEV_Scheme}
\end{figure}

Related existing approaches have considered both fuel minimization and battery management in designing an EM for HEV \cite{Jia,Kim,Khalatbarisoltani,Hu,Mehraban,Li,Ju}. In \cite{Jia}, the authors propose an adaptive MPC for EM of HEV and validate the control through a real-time hardware in loop (HIL) test. The state of charge (SoC) of the battery was regulated around a set-point value, but the impact of this EM design on the battery's health was not considered. In \cite{Kim} the authors incorporated both the battery thermal and degradation models via the so-called \emph{coulomb counting} in their proposed EM strategy. In their work, the impact of the designed EM on battery health was studied. However, the driving duration considered was around 2 hours, which is insufficient to understand the effect of the battery degradation completely. In \cite{Khalatbarisoltani}\cite{Ju} the battery state of health (SoH) model is used directly in the optimization to minimize the battery degradation. However, due to the complex dynamics of the battery capacity function, it is difficult to estimate the battery capacity in order to minimize the degradation. In \cite{Mehraban}, the authors chose to minimize battery power to slow down the battery degradation. However, the approach taken is based on Pontryagin's minimum principle in which not all the system constraints were able to be accounted for. Also, the simulation duration of 200 $\textsf{seconds}$ is insufficient to capture the complex degradation mechanism.

Dynamically positioned hybrid ships as a case study of a hybrid power train, and studying the load split is carried out by numerous researchers in the existing literature \cite{doi:10.1080/20464177.2018.1505584,10.1115/1.4065483,10.1115/1.4048588,DINH201898,BORDIN2019425}. Similarly, hybrid electric aircraft as an example of the hybrid power train is also carried out by few researchers in the existing literature \cite{10.1115/1.4044956,10.1115/GT2023-103131,10.1115/1.4050870}. However, the impact of battery degradation was not explicitly considered in the existing approaches for both dynamically positioned hybrid ships and hybrid electric aircraft. The key contributions in this work are:
\begin{enumerate}
    \item A generalized Euler-Lagrange-based model is presented for HPTs with an example of Road vehicles, Surface vehicles, and Ariel vehicles.
    \item A model predictive energy management strategy considering a heuristic to minimize the battery power, thus minimizing the battery degradation and the efficient operation of the engine around an operating point. Fig.~\ref{Hierarchical} shows the overall developed framework.
    \item A scalable plug-and-play distributed framework is developed, which can be applied to hybrid power trains consisting of $n$ number of power sources.
    \item Realistic simulation experiments are conducted using real drive-cycle data from the U.S. Environment and Protection Agency (EPA) to mimic long-range driving (360 hours) for hybrid electric vehicles and capture evolutionary trends for battery degradation. This demonstrates the effect of the designed model predictive energy management strategy in achieving stated objectives. Similar simulations are performed for dynamically positioned hybrid ships and hybrid electric aircraft.
    \item A \textit{real-time} target simulation of the proposed algorithm in testing the algorithm's real-time capabilities via three test cases of a hybrid electric vehicle, dynamically positioned hybrid ship, and hybrid electric aircraft.
\end{enumerate}

\begin{figure*}[h!]
      \centering
      \includegraphics[width=0.99\columnwidth]{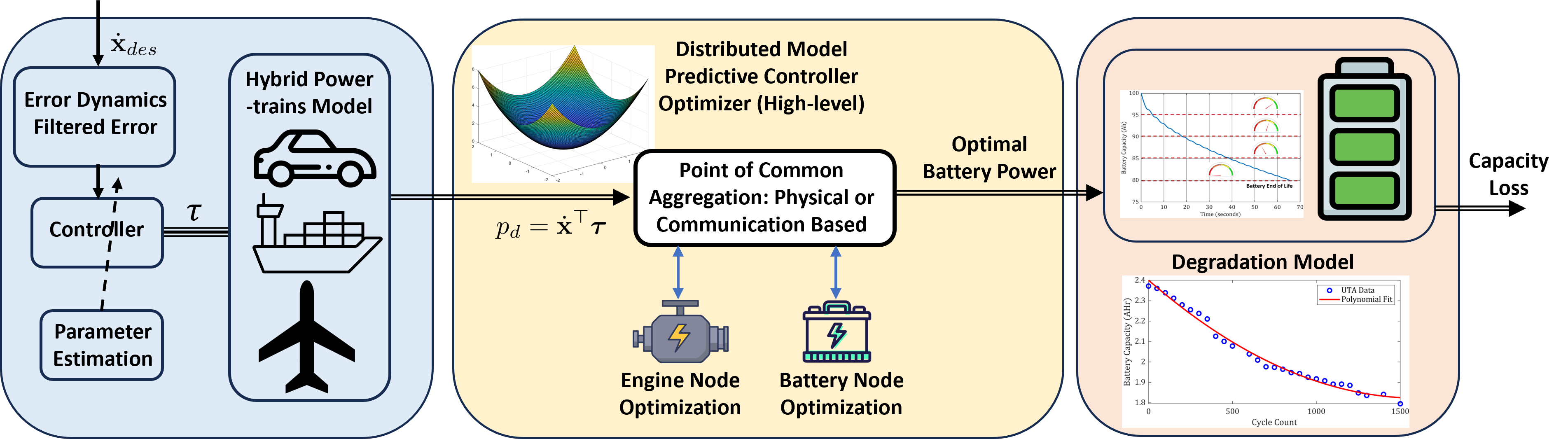}
	 \caption{Distributed optimization implementation scheme}
     \label{Hierarchical} 
\end{figure*}

The rest of the paper is organized as follows: In Section-\ref{sec:notations}, the mathematical notations used throughout the paper are presented. Section-\ref{Sec: Model} contains the model development of the battery and the drive-line dynamics. In Section-\ref{Sec: Model}, the proposed distributed battery health aware model predictive energy management (MPEM) is developed. In Section-\ref{Sec: Simulation}, the developed EM controller is tested via a realistic simulation of the developed model, and the results are presented and discussed. Finally, the conclusions are presented in Section-\ref{Sec: Conclusion}.

\section{Notations} \label{sec:notations}
The set of natural numbers is represented by $\mathbb{N}$ and the set of real numbers denoted by $\mathbb{R}$. A matrix with $n$ rows and $m$ columns is denoted as $\mathbb{R}^{n \times m}$. $\mathbb{R}_+$ denotes the set of positive real numbers. $\mathbb{L}_2$ and $\mathbb{L}_{\infty}$ denote the square-integrable (measurable) and bounded signal spaces. Scalars are denoted by lowercase alphabets (for example $x \in \mathbb{R}$ and $y \in \mathbb{N}$). The notation $\mathcal{P} \subset \mathbb{R}^n$ represents polytopes of the form $\{\mathbf{x} \in \mathbb{R}^n|A_i\mathbf{x}_i \preceq \mathbf{b}_i,A_e\mathbf{x}=\mathbf{b}_e\}$, where the parameters $A_e,\mathbf{b}_e,A_i,\mathbf{b}_i$ can be inferred from problem specific context. Real vectors are represented by the lowercase bold alphabets (i.e. $\textbf{x} \in \mathbb{R}^{n}$). The vector of ones is denoted as $\mathbf{1}$. The vector of zeros is denoted as $\underline{\mathbf{0}}$. $I_n \in \mathbb{R}^{n \times n}$ denotes the identity matrix of dimension $n$. For any vector $\mathbf{x} \in \mathbb{R}^n$, $\|\mathbf{x}\|_2 \triangleq \sqrt{\mathbf{x}^\top\mathbf{x}}$, represents the 2-norm. The symbol $\preceq$ denotes the component-wise inequality i.e. $\mathbf{x} \preceq \mathbf{y}$ is equivalent to $\mathbf{x}_i \leq \mathbf{y}_i$ for $i=1,2,\hdots,n$.

\section{Model Development} \label{Sec: Model}

The hybrid power-train considered in this paper, shown in Fig.~\ref{HEV_Scheme}, is a parallel interconnection of two energy pathways -- one consisting of a unidirectional engine/generator source, while the other is a bidirectional storage/converter resource. This interconnection supplies an electric motor which is a prime mover for the vehicle dynamics. As shown in the figure, this architecture is applicable to a wide variety of modalities including road vehicles, surface vehicles, and aerial vehicles. To this effect, the generalized vehicle dynamics is given by the Euler-Lagrange dynamics \cite{Dixon}
\begin{equation}\label{genaralized_dynamics}
    M(\mathbf{x})\ddot{\mathbf{x}} + V_m(\mathbf{x},\dot{\mathbf{x}})\dot{\mathbf{x}}+G(\mathbf{x})+F(\dot{\mathbf{x}}) = \boldsymbol{\tau},
\end{equation}
where $\mathbf{x}(t), \dot{\mathbf{x}}(t) \in \mathbb{R}^n$ are the generalized position and velocity respectively, $M(\mathbf{x}) \in \mathbb{R}^{n \times n}$ denotes the inertia matrix, $V_m(\mathbf{x},\dot{\mathbf{x}}) \in \mathbb{R}^{n \times n}$ the centripetal-Coriolis matrix, $G(\mathbf{x}) \in \mathbb{R}^{n \times n}$ the gravity vector, $F(\dot{\mathbf{x}}) \in \mathbb{R}^{n \times n}$ the generalized damping torque, and $\boldsymbol{\tau} \in \mathbb{R}^n$ generalized torque input. The dynamics in \eqref{genaralized_dynamics} is well known \cite{spong2020robot} to satisfy the following  properties which are used for subsequent control development:
\begin{enumerate}
    \item{There exists positive real numbers $m_2>m_1>0$ such that $m_1\left\|\mathbf{z}\right\|^2\le \mathbf{z}^\top M(\mathbf{x})\mathbf{z}  \le m_2\left\|\mathbf{z}\right\|^2 $ for all $\mathbf{x},\hspace{2mm}\mathbf{z}\in\mathbb{R}^n$}
    \item {The skew symmetric property $\mathbf{z}^\top\left(\dot{M}(\mathbf{x})-\frac{1}{2}V_m(\mathbf{x},\dot{\mathbf{x}})\right)\mathbf{z} = 0$ holds for all $\mathbf{x},\hspace{2mm}\mathbf{z}\in\mathbb{R}^n$,}
    \item \textit{Linear Parametrization:} For a continuously differentiable trajectory given by the stable dynamics $\dot{\mathbf{x}}_d = f(\mathbf{x}_d)$, where $f:\mathbb{R}^n\rightarrow \mathbb{R}^n$ is a locally lipschitz vector field, the dynamics in \eqref{genaralized_dynamics} can be linearly parametrized as follows
    \begin{align}\label{parameterization}
        V_m(\mathbf{x},\dot{\mathbf{x}}) f(\mathbf{x}_d) &+ G(\mathbf{x}) + F(\dot{\mathbf{x}})+ M(\mathbf{x})\nabla f(\mathbf{x}_d)f(\mathbf{x}_d) =  Y(\mathbf{x},\dot{\mathbf{x}},\mathbf{x}_d)\boldsymbol{\theta},
    \end{align}
    where $Y(\mathbf{x},\dot{\mathbf{x}},\mathbf{x}_d) \in \mathbb{R}^{n \times n_\theta}$ is a \textit{regressor} matrix containing known and measured quantities, and $\boldsymbol{\theta} \in \mathbb{R}^{n_\theta}$ is the corresponding vector of the constant unknown parameters.
\end{enumerate}

Next, we present the design of a generalized control $\boldsymbol{\tau}$ that stabilizes the dynamics in \eqref{genaralized_dynamics} and is subsequently used in determining the power demand. 

\begin{theorem}
    Consider the generalized vehicle dynamics in \eqref{genaralized_dynamics}, together with a differentiable reference trajectory generated by the stable dynamics $\dot{\mathbf{x}}_{d} = f(\mathbf{x}_d)$. Let $$\boldsymbol{\eta}(t) \triangleq \dot{\mathbf{x}}(t) - \dot{\mathbf{x}}_{d}(t) $$ be the associated speed tracking error. Given positive real numbers $k_1, \gamma_1$, the closed loop speed tracking error dynamics is globally asymptotically stable (GAS) under the control law
    \begin{equation}\label{control_law}
        \boldsymbol{\tau}(t) = -k_1 \boldsymbol{\eta}(t) - \gamma_1 Y(t)  \int_{0}^{t} Y(\nu)^\top \boldsymbol{\eta}(\nu) d\nu,
    \end{equation}
    where $Y(t)$ is given in \eqref{parameterization}.
\end{theorem}
\begin{proof}
    Consider the open-loop error dynamics
    \begin{equation}
      M(\mathbf{x}) \boldsymbol{\dot{\eta}} = \boldsymbol{\tau} - V_m(\mathbf{x},\dot{\mathbf{x}}) \boldsymbol{\eta} - Y\boldsymbol{\theta}  
    \end{equation}
   substituting the designed control law in \eqref{control_law}, the following closed-loop dynamics is obtained
    $$M(\mathbf{x})\dot{\boldsymbol{\eta}} = -k_1 \boldsymbol{\eta} - V_m(\mathbf{x},\dot{\mathbf{x}}) \boldsymbol{\eta} - Y\boldsymbol{\tilde{\theta}},$$
    where 
    \begin{align*}
    \boldsymbol{\tilde{\theta}} &= \boldsymbol{\theta} - \boldsymbol{\hat{\theta}}(t),\\
    &= \boldsymbol{\theta} - \gamma_1 Y(t)  \int_{0}^{t} Y(\nu)^\top \boldsymbol{\eta}(\nu) d\nu,
    \end{align*}
    is the parameter estimation error.

    Consider the Lyapunov candidate function
    $$V = \frac{1}{2}\boldsymbol{\eta}^\top M \boldsymbol{\eta} + \frac{1}{2 \gamma_1}\boldsymbol{\tilde{\theta}}^\top \boldsymbol{\tilde{\theta}},$$
    taking the first time derivative along the variables and making use of the skew-symmetric property and substituting the control law in \eqref{control_law} yields,
    $$\dot{V} = -k_1\norm{\boldsymbol{\eta}}^2,$$
    since $\dot{V}$ is negative semi-definite and $V > 0$. Thus, $V \in \mathbb{L}_{\infty}$, which implies $\boldsymbol{\eta} \in \mathbb{L}_{\infty}$, since it is assumed that the reference speed $\dot{\mathbf{x}}_{d} \in \mathbb{L}_{\infty}$, it implies that $\dot{\mathbf{x}} \in \mathbb{L}_{\infty}$. Integrating $\dot{V}$ yields,
    $$V(\infty) - V(0) = -k_1 \int_{0}^{\infty} \norm{\boldsymbol{\eta}(t)}^2 dt,$$
    It follows that $\boldsymbol{\eta} \in \mathbb{L}_2$. Also $\boldsymbol{\eta}$ is uniformly continuous. Thus, invoking Barbalat's lemma \cite{Khalil} it follows that $\lim\limits_{t \rightarrow \infty} \boldsymbol{\eta}(t) = \underline{\mathbf{0}}.$ 
\end{proof}
\begin{remark}
    The resulting generalized power demand that needs to be tracked by the engine ($p_e \in \mathbb{R}$) and the battery ($p_b \in \mathbb{R}$) is given as
    $p_e + p_b = p_d,$ where
    \begin{equation}
        p_d = -\dot{\mathbf{x}}^\top \bigg(k_1 \boldsymbol{\eta}(t) + \gamma_1 Y(t)  \int_{0}^{t} Y(\nu)^\top \boldsymbol{\eta}(\nu) d\nu \bigg).
    \end{equation}
\end{remark}
Fig.~\ref{Speed_Track} shows the derived power demand given the speed profile for road vehicles (hybrid electric cars), surface vehicles (dynamically positioned ships), and aerial vehicles (hybrid electric aircraft).

We make the following remarks concerning special cases of the generalized vehicle dynamics in \eqref{genaralized_dynamics}.
\begin{remark}[Hybrid road vehicles]
    The longitudinal vehicle dynamics \cite{Rajamani} of hybrid road vehicles is of the form in \eqref{genaralized_dynamics} with 
\begin{align*}
    & M(x) = m, \quad & V_m(x,\dot{x}) = 0, \\
    & G(x) = \mu_r mg cos(\phi), \quad & F(\dot{x}) = 0.5 \rho A C_d \dot{x}^2,
\end{align*}
    where $\dot{x}(t) \in \mathbb{R}$ is the vehicle speed [in \textsf{m/s}], $\tau(t) \in \mathbb{R}$ is the effective wheel torque [in \textsf{N-m}], $m \in \mathbb{R}_+$ is the mass of the vehicle [in \textsf{kg}], $\rho_d \in \mathbb{R}_+$ is density of air [in $\textsf{Kg-m}^{-2}$], $C_d \in \mathbb{R}_+$ is the aerodynamic drag coefficient [no units], $A \in \mathbb{R}_+$ is the frontal area of the car [ in $\textsf{m}^2$ ], $\mu_r \in \mathbb{R}_+$ is the coefficient of rolling resistance [no units], $g=9.87m/s^2$ is the acceleration due to gravity, $\phi \in \mathbb{R}_+$ is the road gradient (assumed to be constant), and $R$ is the effective wheel radius in [ \textsf{m} ]. The road vehicle dynamics can be linearly parameterized as 
    \begin{align*}
      &  Y = \begin{bmatrix}
            \dot{x}^2 \\
            \ddot{x}_{d}^\top  \\
            1
        \end{bmatrix}^\top,
      &  \boldsymbol{\theta} = \begin{bmatrix}
          0.5 \rho A C_d \\
          m \\
          \mu_r mg cos(\phi) 
      \end{bmatrix} \in \mathbb{R}^{3}. 
    \end{align*}
\end{remark}
 \begin{remark}[Hybrid surface vehicles]
     For hybrid ships and other thruster-driven hybrid surface vehicles \cite{Dixon}, \cite{Loria}, the dynamics satisfy \eqref{genaralized_dynamics} with
     $ M(\mathbf{x}) = R\overline{M}R^\top$, $V_m(\mathbf{x},\dot{\mathbf{x}}) = R\overline{M}\dot{R}^\top$, $F(\dot{\mathbf{x}}) = RDR^\top\dot{\mathbf{x}}$ and $G(\mathbf{x}) = RK\mathbf{x}$, where
     \begin{align*}
        & R = \begin{bmatrix}
    cos(\psi) & -sin(\psi) & 0 \\ sin(\psi) & cos(\psi) & 0 \\ 0 & 0 & 1
\end{bmatrix}, \\
& \overline{M} = \begin{bmatrix}
    m_{11} & 0 & 0 \\ 0 & m_{22} & m_{23} \\ 0 & m_{23} & m_{33}
\end{bmatrix}, D = \begin{bmatrix}
    d_{11} & 0 & 0 \\ 0 & d_{22} & d_{23} \\ 0 & d_{23} & d_{33}
\end{bmatrix},
     \end{align*}
and $K\in\mathbb{R}^{3\times 3}$ is a diagonal matrix accounting for the mooring forces. $R$ represents the rotation between the Earth and body-fixed coordinate frames, $M \in \mathbb {R}^{3 \times 3}$ denotes the positive-definite and symmetric constant mass-inertia matrix, $D \in \mathbb{R}^{3 \times 3}$ denotes the constant damping matrix. $\mathbf{x} \triangleq [x,y,\psi]^\top$ is the state vector consisting of the translational positions and the yaw angle. For the surface vehicles control input $\boldsymbol{\tau} \in \mathbb{R}^3$ denotes the control force/ control torque required to drive the propeller. The dynamics can be linearly parameterized as \cite{Dixon}
\begin{align*}
    & Y(\mathbf{x},\dot{\mathbf{x}},\ddot{\mathbf{x}}) \in \mathbb{R}^{3 \times 9}, \\
    & \boldsymbol{\theta} = [m_{11};m_{22};m_{23};m_{33};d_{11};d_{22};d_{23};d_{32};d_{3}] \in \mathbb{R}^9
\end{align*}
 \end{remark}
 \begin{remark}[Hybrid aerial Vehicles]
     For aerial vehicles such as hybrid electric aircraft, the longitudinal dynamical model is of the form in \eqref{genaralized_dynamics} with \cite{DOFFSOTTA20206043}
     \begin{align*}
       & M(x) = m_a, \quad &   V_m(x,\dot{x}) = 0, \\
       & G(x) = m_a g sin(\phi_a), \quad & F(\dot{x}) = 0.5 \rho_a S C_d \dot{x}^2,
     \end{align*}
    where $\dot{x}(t) \in \mathbb{R}$ is the aircraft longitudinal velocity [in \textsf{m/s}], $\tau(t) \in \mathbb{R}$ is the effective thrust [in \textsf{N-m}], $m_a \in \mathbb{R}_+$ is the mass of the aircraft [in \textsf{kg}], $\rho_a \in \mathbb{R}_+$ is density of air [in $\textsf{Kg-m}^{-2}$], $C_d \in \mathbb{R}_+$ is the aerodynamic drag coefficient [no units], $S \in \mathbb{R}_+$ is the wing area of the aircraft [ in $\textsf{m}^2$ ], $g=9.87m/s^2$ is the acceleration due to gravity, $\phi_a \in \mathbb{R}_+$ is the aircraft path angle [in \textsf{degree}]. The road vehicle dynamics can be linearly parameterized as
    \begin{align*}
      & Y = \begin{bmatrix}
           \dot{x}^2 \\
           \ddot{x}_{d} \\
           1
      & \end{bmatrix}^\top ,
      \boldsymbol{\theta} = \begin{bmatrix}
          0.5 \rho_a C_d S \\
          m_a \\
          m_a g sin(\phi_a)
        \end{bmatrix} \in \mathbb{R}^3
       \end{align*}
 \end{remark}

 Next, we present the modeling of the energy storage system.

\begin{figure}[t!]
\centering
\centerline{\includegraphics[width=0.95\columnwidth]{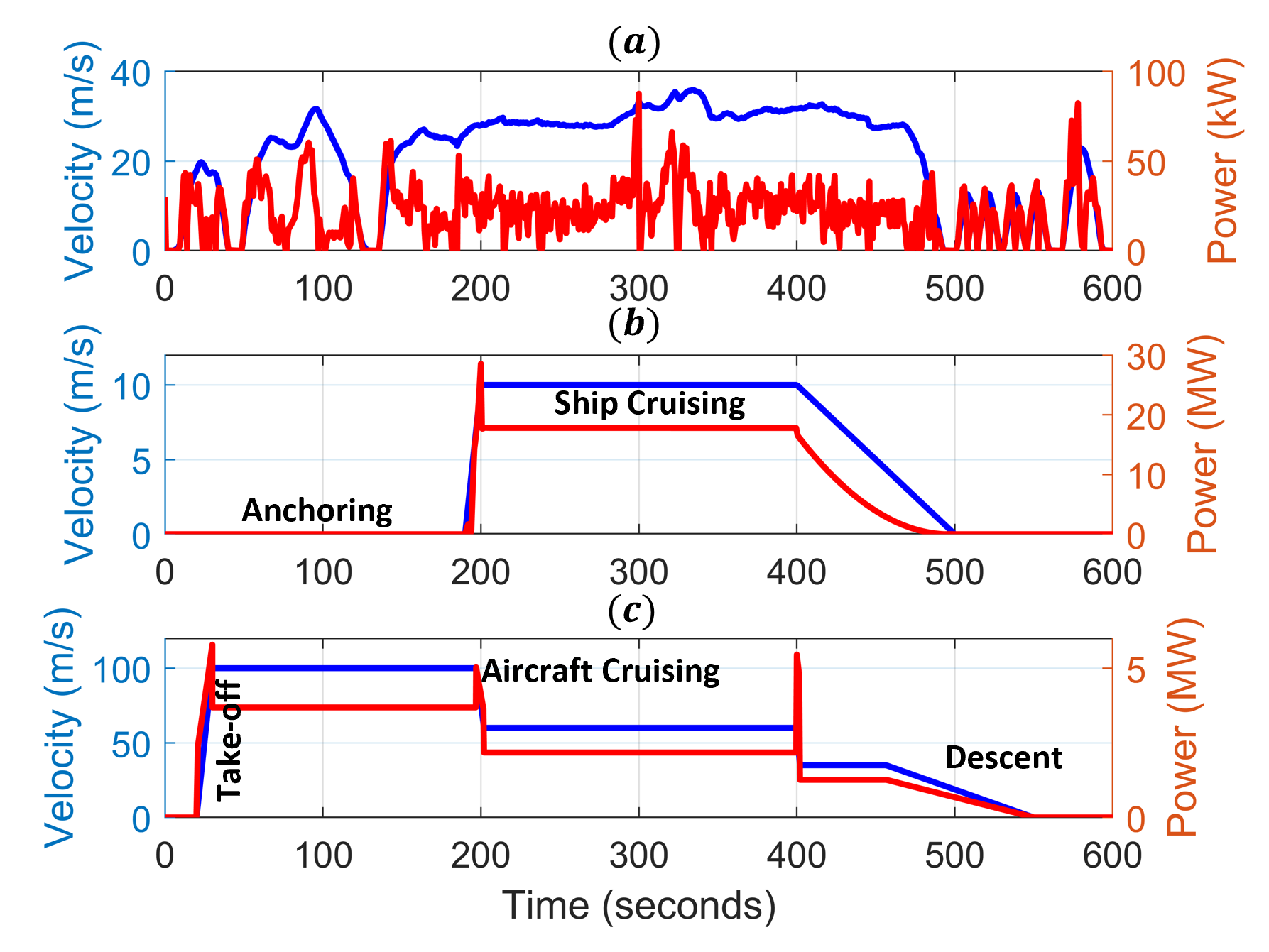}}
\caption{(a) Velocity and power required to track the velocity performance for the hybrid road vehicles, (b) Velocity and power required to track the velocity performance for the hybrid surface vehicles, (c) Velocity and power required to track the velocity performance for the hybrid aerial vehicles.}
\label{Speed_Track}
\end{figure}

\section{Energy Storage System Modeling}
The Energy Storage System (ESS) constitutes a single or multiple hybrid energy storage elements. This might include a collection of flywheels, or super-capacitors. The dynamics of the ESS used in this work and the battery current calculations are based on \cite{10.1115/1.4069921,2021_Mehrzad}. The relationship between the ESS power injected ${p}_{b}$ and the SoC is given as follows:
\begin{equation}\label{SoC_Power}
    q(t) = q_0-\frac{1}{{Q}_{T}v_{b}} \int_{0}^{t} p_{b}(\nu) {d}\nu,
\end{equation}
where ${Q}_{T} \in \mathbb{R}_+$ is the total capacity of the battery in \textsf{AHr}, $v_b \in \mathbb{R}_+$ denotes the battery voltage, which is algebraically expressed as \begin{equation} 
v_b = v_{oc}(q)-r_b i_b, 
\end{equation} 
where $v_{oc}(q)$ is the open-circuit voltage which is a function of the state of charge. The relation between $v_{oc}$ and SoC can be approximated using various functions ranging from lower to higher orders with the linear approximation being the fundamental one given as $v_{oc}=c_1q+c_2$, where $c_1,c_2$ are constants \cite{Weng}. $r_b$ is the internal resistance of the battery and $i_b \in \mathbb{R}$ is the battery current, $q_0 \in [0,1]$ is the initial SoC of the battery. Although the SoC dynamics above does not explicitly capture the physical restriction $q(t)\in\left[0,\hspace{1mm}1\right]\text{ for all }t\in\mathbb{R}_+$, we impose this explicit constraint in the underlying optimization problem of the proposed energy management strategy. As a result, it is ignored in the dynamics without loosing fidelity. The SoC dynamics in (\ref{SoC_Power}) is discretized, with a sampling time ${T_s} \in \mathbb{R}_+$, to obtain the following discrete dynamics:

\begin{equation}\label{SoC_Power_Disctrete}
    \sum_{{k}=0}^{N} p_{{b}_{k}} = \frac{{Q}_{T}v_{b}}{T_s}(q_\text{0}-q_N),
\end{equation}
where $q_{N}$ is the SoC at the time instant $N{T_s} \in \mathbb{R_+}$, and $N$ is the number of discrete time samples in the prediction horizon.

The battery degradation model is based on the Arrhenius equation and uses the $Ah-$throughput $\displaystyle \int_{0}^{t}\left|i_b(\nu)\right|d\nu$ as a metric to evaluate the battery state of health (SoH), where $i_b(t) \in \mathbb{R}$ is the current drawn from the battery (positive in the discharge direction and vice-versa). The \emph{capacity loss} is given as \cite{SONG2018433}:
\begin{equation}
    Q_L(t) =  e^{\frac{-\zeta_2+TC_{r}}{RT}}\int_0^{t}|i_b(\nu)|^{\frac{1}{2}}d\nu,
\end{equation}
where $T \in \mathbb{R}_+$ is the baseline temperature (in \textsf{Kelvin}) of the battery where it is most efficient and $C_r \in \mathbb{R}_+$ is the C-rate of the battery. The capacity loss (\%) is given as follows: $$\Delta Q \% = \frac{Q_T-Q_{L}(t)}{Q_T}\times 100,$$
where $Q_{L}$ is the \emph{capacity loss} of the ESS in \textsf{ampere-hour}. $\Delta Q$ only captures the \textit{capacity loss} during the ESS operation.

We have presented the model development required and next we present the proposed energy management development.

\section{Proposed Energy Management Method}\label{Sec: Control_Development}
The MPEM scheme is given in the form of the optimization problem:
\begin{equation}\label{MPC_No_Equality}
    \begin{aligned}
    \Minimize_{p_{e_k},p_{b_k},q_k} \quad & \sum_{k=1}^{h} \bigg(\frac{\alpha}{2}({p_e}_k+{p_b}_k-\hat{{p}}_{d_k})^2 +C_e({p_e}_k)+C_b({p_b}_k)\bigg) \\
\SubjectTo  \hspace{2mm} &  \underline{p}_e \preceq p_{e_k} \preceq \overline{p}_e, \hspace{1mm} \forall k=1,2,\hdots,h, \\ 
     & \underline{p}_b \preceq p_{b_k} \preceq \overline{p}_b, \hspace{1mm} \forall k=1,2,\hdots,h, \\
   & \underline{q} \preceq q_{k} \preceq \overline{q}, \hspace{1mm} \forall k=1,2,\hdots,h, \\
   & |p_{e_k} - p_{e_{k-1}}| \preceq r_e, \hspace{1mm} \forall k=1,2,\hdots,h, \\
   & |p_{b_k} - p_{b_{k-1}}| \preceq r_e, \hspace{1mm} \forall k=1,2,\hdots,h, \\
   & \sum_{{k}=1}^{h} p_{{b}_{k}} = \frac{{Q}_{T}v_{b}}{T_s}(q_\text{0}-q_h),
    \end{aligned}
\end{equation}
where $\left[\begin{array}{cccc} {p_{e_1},p_{e_2},\hdots,p_{e_h}} \end{array}\right]^\top \in \mathbb{R}^h$ is the engine power profile over a prediction horizon of length $h$, $\left[\begin{array}{cccc} {p_{b_1},p_{b_2},\hdots,p_{b_h}} \end{array}\right]^\top \in \mathbb{R}^h$ is the battery power profile over a prediction horizon of length $h$, $\left[\begin{array}{cccc} {\hat{p}_{d_1},\hat{p}_{d_2},\hdots,\hat{p}_{d_h}} \end{array}\right]^\top \in \mathbb{R}^h$ is the power forecast over the horizon of the length $h$ assumed to be held constant for the length of the horizon, $q_h$ is the SoC at the end of each horizon, $\alpha \in \mathbb{R}_+$ is the weighting associated with the power tracking objective. $C_e:\mathbb{R}^h \to \mathbb{R}_+$ is the cost function associated with the engine efficiency. $C_b:\mathbb{R}^h \to \mathbb{R}_+$ is the cost function associated with the battery degradation and health. The constraint set is a \textit{polytope} comprised of an affine equality constraints and an affine inequality constraints. For the ease of exposition, we group the respective engine, and battery constraints as $\mathcal{P}_e, \mathcal{P}_b \subset \mathbb{R}^h$ representing individual constraint sets for each of the engine and the battery capturing their ramp-rate limitations and box constraints and relevant system dynamic constraints. The goal here is to split the optimization problem in (\ref{MPC_No_Equality}) into two local optimization problems, one each for the engine and the battery. First, we write (\ref{MPC_No_Equality}) as:
\begin{equation}\label{MPC_HEV_ADMM_REWRITTEN}
\begin{aligned}
\Minimize_{y_{e_k},y_{b_k},p_{e_k},p_{b_k}} \quad & \sum_{k=1}^{h} \bigg(\frac{\alpha}{2}({y_e}_k+{y_b}_k-\hat{p}_{d_k})^2+ C_e({p_e}_k)+C_b({p_b}_k)\bigg)   \\
\SubjectTo \quad & p_{e_k}  = y_{e_k}, \hspace{2mm} \forall k=1,2,\hdots,h, \\
 \quad & p_{b_k} = y_{b_k}, \hspace{2mm} \forall k=1,2,\hdots,h, \\
\quad &  {p_e}_k \in \mathcal{P}_e, \hspace{3mm} \forall k=1,2,\hdots,h, \\ \quad & {p_b}_k  \in \mathcal{P}_b, \hspace{3mm} \forall k=1,2,\hdots,h,
\end{aligned}
\end{equation}
where $y_{e_k},y_{b_k} \in \mathbb{R}$ are the dummy variables introduced to split the problem. The optimization problem in (\ref{MPC_HEV_ADMM_REWRITTEN}) is written more compactly as
\begin{equation}\label{MPC_HEV_ADMM_EQUIV}
\begin{aligned}
\Minimize_{\mathbf{y}_e,\mathbf{y}_b,\mathbf{p}_e,\mathbf{p}_b} \quad &  \frac{\alpha}{2}\norm{{\mathbf{y}_e}+{\mathbf{y}_b}-\hat{\mathbf{p}}_{d}}_2^2+C_e({\mathbf{p}_e})+C_b({\mathbf{p}_b})   \\
\SubjectTo \quad & \mathbf{p}_{e}  = \mathbf{y}_{e}, 
   \mathbf{p}_{b} = \mathbf{y}_{b},  \\
 &  {\mathbf{p}_e} \in \mathcal{P}_e,  {\mathbf{p}_b}  \in \mathcal{P}_b,
\end{aligned}
\end{equation}
where $\mathbf{y}_e \triangleq \left[\begin{array}{ccccc} {y_{e_1},y_{e_{2}},\hdots,y_{e_{h}}} \end{array}\right]^\top \in \mathbb{R}^h$, $\mathbf{y}_b \triangleq \left[\begin{array}{cccc} {y_{b_1},y_{b_{2}},\hdots,y_{b_{h}}} \end{array}\right]^\top \in \mathbb{R}^h$, $\sum_{k=1}^{h}C_e(p_{e_k})\triangleq C_e(\mathbf{p}_e)$ and $\sum_{k=1}^{h}C_b(p_{b_k})\triangleq C_b(\mathbf{p}_b)$. for a given parameter $\rho>0$, the augmented Lagrangian for the optimization problem in (\ref{MPC_HEV_ADMM_EQUIV}) is given as \cite{boyd}:
\begin{align*}\nonumber
 \mathcal{L}_\rho(\mathbf{y}_e,\mathbf{y}_b,\mathbf{p}_e,\mathbf{p}_b,\boldsymbol{\lambda}_e,\boldsymbol{\lambda}_b) = \frac{\alpha}{2}\norm{{\mathbf{y}_e}+{\mathbf{y}_b}-\hat{\mathbf{p}}_{d}}_2^2+ C_e({\mathbf{p}_e})\\+C_b({\mathbf{p}_b}) +\boldsymbol{\lambda}_e^\top (\mathbf{p}_e-\mathbf{y}_e)+\boldsymbol{\lambda}_b^\top (\mathbf{p}_b-\mathbf{y}_b)\\+\frac{\rho}{2}\norm{\mathbf{p}_e-\mathbf{y}_e}_2^2+\frac{\rho}{2}\norm{\mathbf{p}_b-\mathbf{y}_b}_2^2+\mathcal{I_\mathcal{P}}_e+\mathcal{I_\mathcal{P}}_b,
\end{align*}
where $\mathcal{I_\mathcal{P}}_e$ and $\mathcal{I_\mathcal{P}}_b$ are the indicator functions for the inclusion constraints  $\mathbf{p}_e \in \mathcal{P}_e$ and $\mathbf{p}_b \in \mathcal{P}_b$ respectively. Equivalently, the Lagrangian in the scaled form is given as \cite{boyd}
\begin{align*}\nonumber\label{Scaled_ADMM_Rewritten}
    \mathcal{L}_\rho(\mathbf{y}_e,\mathbf{y}_b,\mathbf{p}_e,\mathbf{p}_b,\mathbf{u}_e,\mathbf{u}_b) = \frac{\alpha}{2}\norm{{\mathbf{y}_e}+{\mathbf{y}_b}-\hat{\mathbf{p}}_{d}}_2^2+ C_e({\mathbf{p}_e})\\+C_b({\mathbf{p}_b})+\frac{\rho}{2}\norm{\mathbf{p}_e-\mathbf{y}_e+\mathbf{u}_e}_2^2\\+\frac{\rho}{2}\norm{\mathbf{p}_b-\mathbf{y}_b+\mathbf{u}_b}_2^2-\frac{\rho}{2}\norm{\mathbf{u}_e}_2^2-\frac{\rho}{2}\norm{\mathbf{u}_b}_2^2+\mathcal{I_\mathcal{P}}_e+\mathcal{I_\mathcal{P}}_b,
    \end{align*}
where $\mathbf{u}_e\triangleq{\frac{1}{\rho}}\boldsymbol{\lambda}_e$ and $\mathbf{u}_b\triangleq{\frac{1}{\rho}}\boldsymbol{\lambda}_b$. 

 Consequently, following the Gauss-Seidel alternating minimization, one obtains the following successive iterative scheme \cite{boyd}:
\begin{subequations}\label{HEV_ADMM_MPC}
\begin{align}
\mathbf{y}_e^{t+1} &= \argmin_{\mathbf{y}_e} \bigg\{\frac{\alpha}{2}\norm{{\mathbf{y}_e}+{\mathbf{y}_b}-\hat{\mathbf{p}}_{d}}_2^2+\frac{\rho}{2}\norm{\mathbf{p}_e^t-\mathbf{y}_e+\mathbf{u}_e^t}_2^2\bigg\}, \\
\mathbf{y}_b^{t+1} &= \argmin_{\mathbf{y}_b}  \bigg\{\frac{\alpha}{2}\norm{{\mathbf{y}_e}+{\mathbf{y}_b}-\hat{\mathbf{p}}_{d}}_2^2+\frac{\rho}{2}\norm{\mathbf{p}_b^t-\mathbf{y}_b+\mathbf{u}_b^t}_2^2\bigg\},\\
\mathbf{p}_e^{t+1} &= \argmin_{\mathbf{p}_e \in \mathcal{P}_e}  \bigg\{C_e(\mathbf{p}_e)+\frac{\rho}{2}\norm{\mathbf{p}_e-\mathbf{y}_e^{t+1}+\mathbf{u}_e^t}_2^2\bigg\}, \\
\mathbf{p}_b^{t+1} &= \argmin_{\mathbf{p}_b \in \mathcal{P}_b} \bigg\{C_b(\mathbf{p}_b)+\frac{\rho}{2}\norm{\mathbf{p}_b-\mathbf{y}_b^{t+1}+\mathbf{u}_b^t}_2^2\bigg\}, \\
 \mathbf{u}_e^{t+1} &= \mathbf{u}_e^t+\mathbf{p}_e^{t+1}-\mathbf{y}_e^{t+1}, \\
  \mathbf{u}_b^{t+1} &= \mathbf{u}_b^t+\mathbf{p}_b^{t+1}-\mathbf{y}_b^{t+1}.
\end{align}
\end{subequations}

The convergence of iterations in (\ref{HEV_ADMM_MPC}) is a well-studied topic and is available in \cite{boyd}. The optimization problems in (\ref{HEV_ADMM_MPC}a) and (\ref{HEV_ADMM_MPC}b) are unconstrained quadratic programs and a closed-form solution can be readily obtained.

\begin{theorem}\label{thm:closed_form}
The optimization steps in (\ref{HEV_ADMM_MPC}a) and (\ref{HEV_ADMM_MPC}b) are equivalent to:
\begin{align*}
& \mathbf{y}_e^{t+1}=\mathbf{p}_e^t+\mathbf{u}_e^t+\mathbf{a}^t, \\
& \mathbf{y}_b^{t+1}=\mathbf{p}_b^t+\mathbf{u}_b^t+\mathbf{a}^t,
\end{align*}
where $$\mathbf{a}^t = \frac{\alpha}{2\alpha+\rho}\bigg(\hat{\mathbf{p}}_d-(\mathbf{p}_e^t+\mathbf{p}_b^t+\mathbf{u}_e^t+\mathbf{u}_b^t)\bigg).$$
\end{theorem}
\begin{proof} The first order optimality conditions for (\ref{HEV_ADMM_MPC}a) and (\ref{HEV_ADMM_MPC}b) are:
\begin{subequations}\label{Optimality}
\begin{align}
    \alpha\bigg(\mathbf{y}_{e}^*+\mathbf{y}_{b}^*-\hat{\mathbf{p}}_{d}\bigg)-\rho\bigg(\mathbf{p}_e^t-\mathbf{y}_e^*+\mathbf{u}_e^t\bigg) = 0, \\
    \alpha\bigg(\mathbf{y}_{e}^*+\mathbf{y}_{b}^*-\hat{\mathbf{p}}_{d}\bigg)-\rho\bigg(\mathbf{p}_b^t-\mathbf{y}_b^*+\mathbf{u}_b^t\bigg) = 0.
\end{align}
\end{subequations}
Rewriting the above equation in the matrix form combining both (\ref{Optimality}a) and (\ref{Optimality}b) yields: 
\begin{align}
   \bigg(\frac{\alpha}{\rho}\begin{bmatrix} I_h \\ I_h \end{bmatrix}\begin{bmatrix}I_h \\ I_h \end{bmatrix}^\top+\begin{bmatrix}I_h \\ I_h \end{bmatrix}\bigg)\begin{bmatrix} \mathbf{y}_e^* \\ \mathbf{y}_b^* \end{bmatrix} = \frac{\alpha}{\rho} \begin{bmatrix} \hat{\mathbf{p}}_d \\ \hat{\mathbf{p}}_d \end{bmatrix}+\begin{bmatrix} \mathbf{p}_e^t \\ \mathbf{p}_b^t \end{bmatrix}+\begin{bmatrix} \mathbf{u}_e^t \\ \mathbf{u}_b^t \end{bmatrix}.
\end{align}
Using the matrix inversion lemma, it follows that
\begin{align*}\label{Agg_sub}
    \mathbf{y}_e^* = \mathbf{p}_e^t + \mathbf{u}_e^t  +\frac{\alpha}{2\alpha+\rho}\bigg(\hat{\mathbf{p}}_d-(\mathbf{p}_e^t+\mathbf{p}_b^t+\mathbf{u}_e^t+\mathbf{u}_b^t)\bigg) \\
    \mathbf{y}_b^* = \mathbf{p}_b^t + \mathbf{u}_b^t  +\underbrace{\frac{\alpha}{2\alpha+\rho}\bigg(\hat{\mathbf{p}}_d-(\mathbf{p}_e^t+\mathbf{p}_b^t+\mathbf{u}_e^t+\mathbf{u}_b^t)\bigg)}_{ \triangleq \mathbf{a}^t}
\end{align*}
\end{proof}

From Theorem~\ref{thm:closed_form}, it is seen that $\mathbf{a}^t$ acts as an aggregator or a communicator between the engine and the battery nodes updating the dual variable. Thus each node implements the following successive optimization problems at each iteration:
\begin{subequations}\label{Optimization_Updates}
\begin{align}
\mathbf{y}_e^{t+1} &= \mathbf{p}_e^t + \mathbf{u}_e^t  +\mathbf{a}^t, \\
\mathbf{y}_b^{t+1} &= \mathbf{p}_b^t + \mathbf{u}_b^t  +\mathbf{a}^t,\\
\mathbf{p}_e^{t+1} &= \argmin_{\mathbf{p}_e \in \mathcal{P}_e}  \bigg\{C_e(\mathbf{p}_e)+\frac{\rho}{2}\norm{\mathbf{p}_e-\mathbf{y}_e^{t+1}+\mathbf{u}_e^t}_2^2\bigg\}, \\
\mathbf{p}_b^{t+1} &= \argmin_{\mathbf{p}_b \in \mathcal{P}_b}  \bigg\{C_b(\mathbf{p}_b)+\frac{\rho}{2}\norm{\mathbf{p}_b-\mathbf{y}_b^{t+1}+\mathbf{u}_b^t}_2^2\bigg\}, \\
 \mathbf{u}_e^{t+1} &= \mathbf{u}_e^t+\mathbf{p}_e^{t+1}-\mathbf{y}_e^{t+1}, \\
  \mathbf{u}_b^{t+1} &= \mathbf{u}_b^t+\mathbf{p}_b^{t+1}-\mathbf{y}_b^{t+1}.
\end{align}
\end{subequations}
\begin{proposition}\label{prop:optimization_updates}
    \textit{The iterative scheme in (\ref{Optimization_Updates}a-\ref{Optimization_Updates}f) is equivalent to the  dynamic updates}:
  \begin{subequations} 
   \begin{align}\label{eqn:eng_opt}
   \begin{array}{rl}
   \mathbf{p}_e^{t+1} &= \argmin\limits_{\mathbf{p}_e\in\mathcal{P}_e} \bigg\{C_e(\mathbf{p}_e)+\frac{\rho}{2}\norm{\mathbf{p}_e-\mathbf{p}_e^t-\mathbf{a}^t}_2^2\bigg\}, \\ 
   \mathbf{z}_e^{t+1} &= 2\mathbf{p}_e^{t+1}-\mathbf{p}_e^t-\mathbf{a}^t, 
   \end{array}\\\nonumber\\\label{eqn:bat_opt}
   \begin{array}{rl}
   \mathbf{p}_b^{t+1} &= \argmin\limits_{\mathbf{p}_b\in\mathcal{P}_b} \bigg\{C_b(\mathbf{p}_b)+\frac{\rho}{2}\norm{\mathbf{p}_b-\mathbf{p}_b^t-\mathbf{a}^t}_2^2\bigg\}, \\ 
   \mathbf{z}_b^{t+1} &= 2\mathbf{p}_b^{t+1}-\mathbf{p}_b^t-\mathbf{a}^t, 
   \end{array}\\\nonumber\\\label{eq:agg}
   \mathbf{a}^t = \frac{\alpha}{2\alpha + \rho}\left(\hat{\mathbf{p}}_d - \mathbf{z}_e^t - \mathbf{z}_b^t\right),\hspace{2cm}
   \end{align}
   where $\mathbf{z}^t \triangleq \mathbf{p}^t + \mathbf{u}^t$.
  \end{subequations}
\end{proposition}
\begin{remark}
  Proposition~\ref{prop:optimization_updates} expresses the iterative scheme in (\ref{Optimization_Updates}a-\ref{Optimization_Updates}f) to distributed dynamic updates where \eqref{eqn:eng_opt} and \eqref{eqn:bat_opt} are parallel operations for the engine and battery nodes optimization respectively, and \eqref{eq:agg} is the aggregator that couples them via the dual updates. This is shown schematically in Fig~\ref{HEV_Optimization}.
\end{remark}
\begin{proof} Substituting (\ref{Optimization_Updates}a) in (\ref{Optimization_Updates}e) and (\ref{Optimization_Updates}b) in (\ref{Optimization_Updates}f) we get,
    \begin{align*}\mathbf{z}_e^{t+1}&=\mathbf{p}_e^{t+1}+\mathbf{u}_e^{t}+\mathbf{p}_e^{t+1}-\mathbf{p}_e^t-\mathbf{u}_e^t-\mathbf{a}^t, \\ &= 2\mathbf{p}_e^{t+1}-\mathbf{p}_e^t-\mathbf{a}^t. \end{align*}
    \begin{align*}\mathbf{z}_b^{t+1}&=\mathbf{p}_b^{t+1}+\mathbf{u}_b^{t}+\mathbf{p}_b^{t+1}-\mathbf{p}_b^t-\mathbf{u}_b^t-\mathbf{a}^t, \\ &= 2\mathbf{p}_b^{t+1}-\mathbf{p}_b^t-\mathbf{a}^t. \end{align*}
    substituting (\ref{Optimization_Updates}a) in (\ref{Optimization_Updates}c) and (\ref{Optimization_Updates}b) in (\ref{Optimization_Updates}d) we get:
    \begin{align*}
        \mathbf{p}_e^{t+1} = \argmin_{\mathbf{p}_e \in \mathcal{P}_e}\left\{C_e(\mathbf{p}_e)+\frac{\rho}{2}\norm{\mathbf{p}_e-\mathbf{p}_e^t-\mathbf{a}^t}_2^2\right\} \\
        \mathbf{p}_b^{t+1} = \argmin_{\mathbf{p}_b \in \mathcal{P}_b}\left\{C_b(\mathbf{p}_b)+\frac{\rho}{2}\norm{\mathbf{p}_b-\mathbf{p}_b^t-\mathbf{a}^t}_2^2\right\},
    \end{align*}
    thus completing the proof.
\end{proof}

Next, the individual nodal level optimization problems are given, with explicitly specified choices for the functions $C_e$, $C_b$, and node-specific inclusion constraints for the respective nodes.

\subsection{Optimization Problem at Engine Node}
The optimization problem at the engine node is given as:
\begin{equation} \label{Engine_Node}
\begin{aligned}
\Minimize_{\mathbf{p}_e} \quad &  \frac{\beta}{2}\norm{\mathbf{p}_e-\mathbf{p}_e^r}_2^2+\frac{\rho}{2}\norm{\mathbf{p}_e-\mathbf{p}_e^t-\mathbf{a}^t}_2^2 \\
\SubjectTo \quad & \underline{\mathbf{p}}_e \preceq \mathbf{p}_e \preceq \overline{\mathbf{p}}_e, \\
\quad & \begin{bmatrix}  D \\ -D  \end{bmatrix} \mathbf{p}_e \preceq r_e\mathbf{1},
\end{aligned} 
\end{equation}
where 
$$D \triangleq
I_h - 
\begin{pNiceArray}{c|c}
  \underline{\mathbf{0}}^\top & 0 \\
  \hline
  I_{h-1} & \underline{\mathbf{0}}
\end{pNiceArray} \in \mathbb{R}^{h \times h},
$$ 
is the point-wise difference operator through the prediction horizon, $r_e$ is the ramp-rate limit of the engine pertaining to the engine torque limitation, $\mathbf{p}_e^r$ is the most efficient operating point of the engine chosen based on the engine efficiency data, and $\overline{\mathbf{p}}_e$ and $\underline{\mathbf{p}}_e$ are the upper and lower limits on engine power respectively. Here, the choice $C_e(\mathbf{p}_e) =\frac{\beta}{2}\norm{\mathbf{p}_e-\mathbf{p}_e^r}_2^2$ penalizes the deviation from the known efficient operation of the engine. The penalty factor $\beta$ is a tuning parameter to weigh the relative importance of the engine efficiency cost against the algorithm convergence speed.
\subsection{Optimization problem at Battery Node}
The optimization problem at the battery node is given as:
\begin{equation}\label{Battery_Node}
\begin{aligned}
\Minimize_{\mathbf{p}_b,\mathbf{q}} \quad  & \frac{\gamma}{2}\norm{\mathbf{p}_b}_2^2+\frac{\rho}{2}\norm{\mathbf{p}_b-\mathbf{p}_b^t-\mathbf{a}^t}_2^2\\
\SubjectTo \quad & \sum_{k=1}^{h}\mathbf{p}_b = \kappa(\mathbf{q}_0-\mathbf{q}_h), \\
&  \underline{\mathbf{p}}_b \preceq \mathbf{p}_b \preceq \overline{\mathbf{p}}_b,\\ 
& \underline{\mathbf{q}} \preceq \mathbf{q} \preceq \overline{\mathbf{q}}, \\
& \begin{bmatrix}  D \\ -D  \end{bmatrix} \mathbf{p}_b \preceq r_b\mathbf{1},
\end{aligned} 
\end{equation}
where $C_b(\mathbf{p}_b) = \frac{\gamma}{2}\norm{\mathbf{p}_b}^2$ is the cost function minimizing the battery power and $\gamma$ is the penalty factor that is tuned to weight the importance of battery usage, $q_0$ is the initial state of charge measurement from the system.  $\kappa = Q_Tv_b/T_s$,  $\overline{\mathbf{p}}_b$ and $\underline{\mathbf{p}}_b$ represent upper and lower limits on battery power respectively, $\mathbf{q}_0$ and $\mathbf{q}_h$ represent the initial SoC and the SoC at the end for the length of the horizon, $\underline{\mathbf{q}}$ and $\overline{\mathbf{q}}$ represent the upper and lower SoC limitations. $r_b$ represents the ramp-rate limitation of the battery. 

\begin{figure}[t!]
      \centering
      \includegraphics[width=0.75\columnwidth]{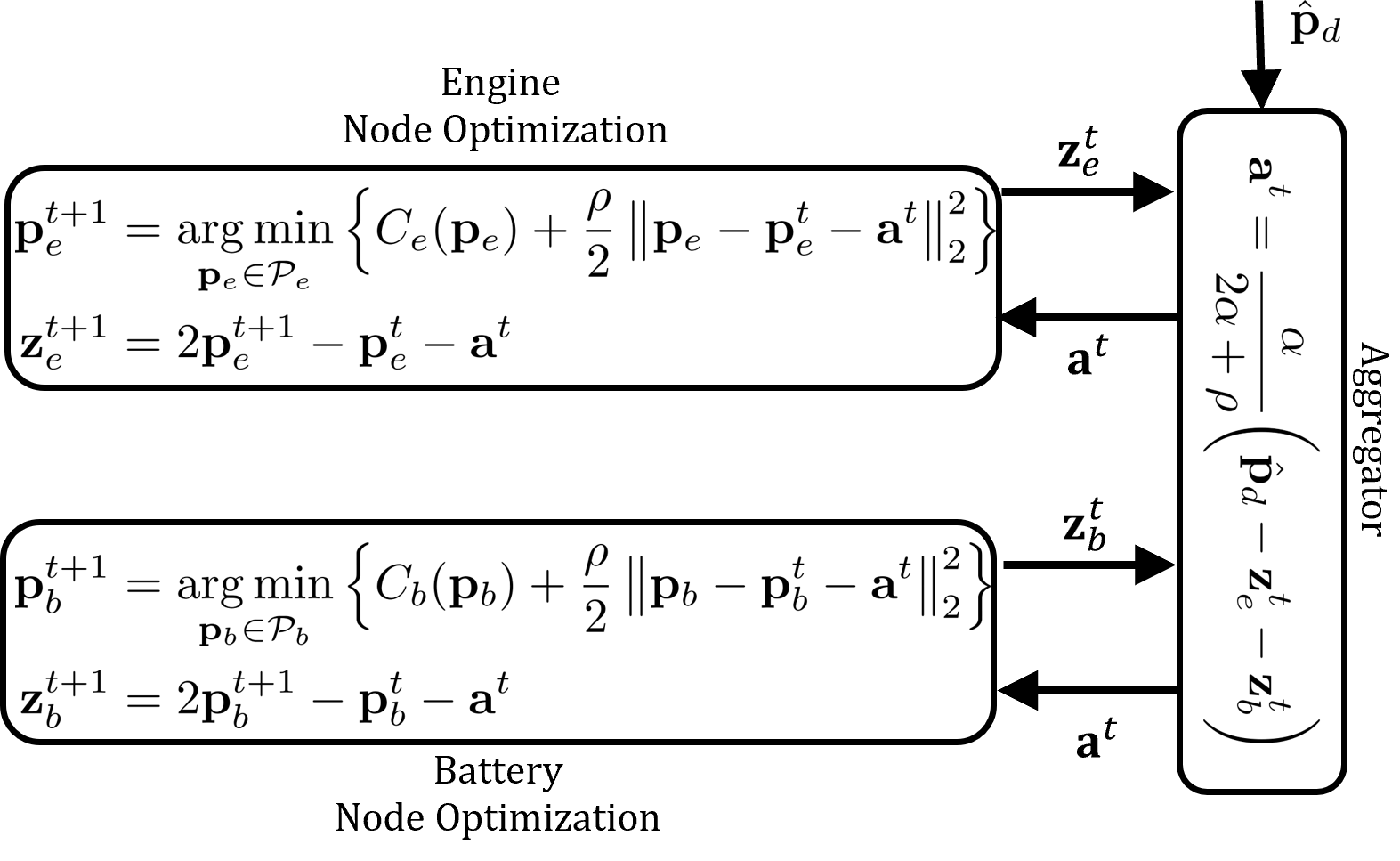}
	 \caption{Distributed optimization implementation scheme}
     \label{HEV_Optimization} 
\end{figure}

\begin{figure*}[t!]
\centering
\centerline{\includegraphics[width=0.99\columnwidth]{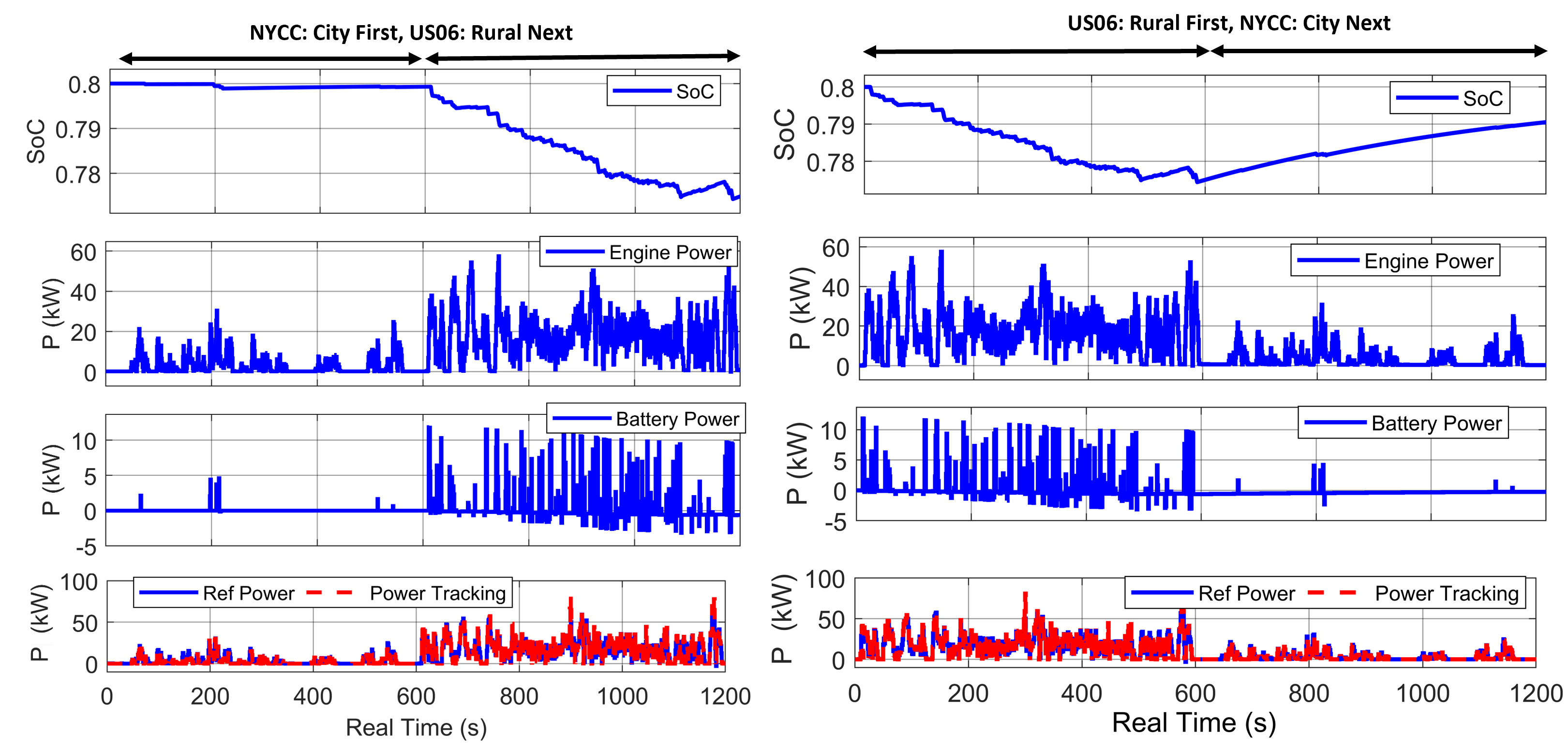}}
\caption{State of charge, engine optimal power, battery optimal power, and the power tracking for one simulation run for US06 drive-cycle and NYCC drive-cycle swapping the driving profiles.}
\label{Power_Track}
\end{figure*}

\begin{figure}[t!]
\centering
\includegraphics[width=0.85\columnwidth]{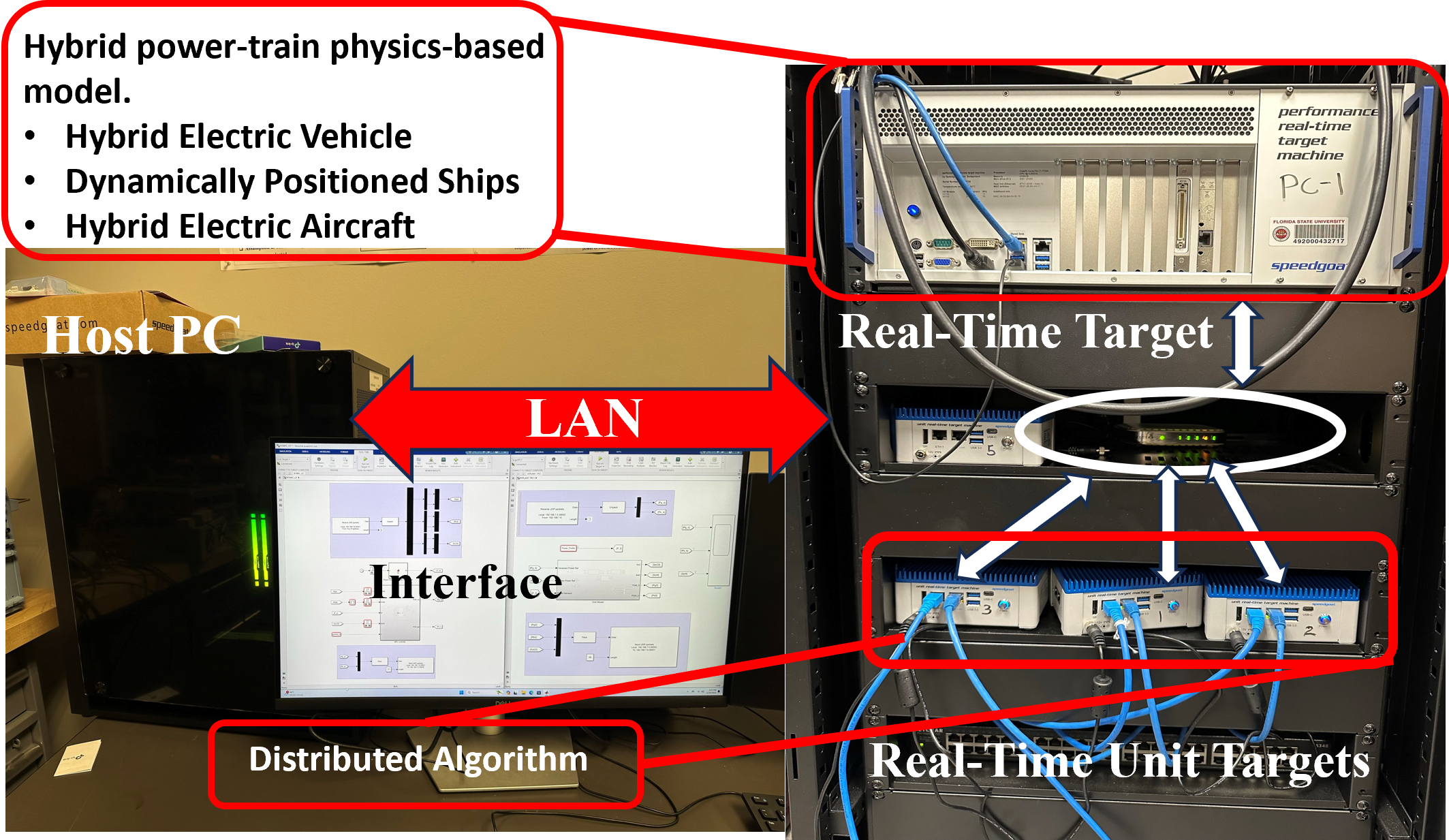}
\caption{Real-time Setup}
\label{RT_Setup}
\end{figure}

\section{Real-Time Numerical Simulation - Hybrid Electric Vehicle, dynamically positioned hybrid ship and hybrid electric aircraft}\label{Sec: Simulation}

We present three real-time simulation test cases using hybrid electric vehicles (hybrid road vehicles), dynamically positioned ships (hybrid surface vehicles), and hybrid electric aircraft (hybrid aerial vehicles) to test the developed distributed algorithm. The real-time simulation setup consists of a Host PC connected to a real-time performance target machine (SPEEDGOAT in this case) and three other smaller real-time unit target machines over a local area network (LAN) connection (1000 MBPS) receive/transmit link speed through a switch. Fig.~\ref{RT_Setup} shows the setup in the lab. The longitudinal dynamics are implemented in the target machine and the distributed algorithm is deployed on the unit targets. The computational capabilities of the devices are provided in Table-\ref{Comp_info}. 

\begin{table}[h!]
\centering
\caption{Computational Info for RT-Setup}
\label{Comp_info}
\resizebox{0.9\columnwidth}{!}{\begin{tabular}{c|c|c}
\hline \hline
\textbf{Device Name}    & \textbf{CPU} & \textbf{RAM}  \\
 \hline \hline
Host PC               & Intel Core i9 3.20GHz 24-core & 64GB   \\
Target    & Intel Core 3.6GHz, 8-core          & 32GB \\
Unit Target-1  & Intel Atom x5-E3940 1.6GHz 4-core & 4GB \\
\hline
\end{tabular}}
\end{table}

\begin{table}[h!]
\centering
\caption{Table Showing Rated Powers and Limits of the Engine and the Battery}
\label{tab:rated}
\resizebox{0.7\columnwidth}{!}{\begin{tabular}{c|c}
\hline \hline
\textbf{Parameter}   & \textbf{Value}  \\
\hline \hline
Rated Engine Power & 150\textsf{hp} or 112\textsf{kW}   \\
Rated Battery Power & 15\textsf{kW}  \\
Desired Operating Engine Power  & 75\textsf{kW}  \\
Engine Upper Power Limit    & 100\textsf{kW}\\ 
Battery Upper Power Limit & 14\textsf{kW} \\
Battery Lower Power Limit & -14\textsf{kW} \\
State of Charge Lower Limit & 0.5 \\
State of Charge Upper Limit & 0.8 \\
 \hline
\end{tabular}}
\end{table}

\subsection{Hybrid Electric Vehicle}
The proposed algorithm along with the developed model of a HEV is tested in a real-time MATLAB-Simulink environment. The optimization problems presented in (\ref{Engine_Node}) and (\ref{Battery_Node}) are used in the simulation process. The parameters for the HEV drive-line are chosen as $m$ = 1400\textsf{kg}, $r$ = 0.2\textsf{m}, $\rho_d$ = 1.225$\textsf{Kg-m}^{-2}$, $C_d$ = 0.35, $A$ = 1.93$\textsf{m}^2$, $\mu_{rr}$ = 0.03. The optimization parameters which are fixed across the simulations are: $\alpha=1000$, $\beta=1$, $\rho=0.1$, and $h=5$\textsf{seconds}. The simulation is run at a fixed time-step of $1$\textsf{ms}, with the optimization problem running every $1$\textsf{second}. The power ratings and the power limits of the engine and the battery are given in Table \ref{tab:rated}.

\begin{figure}[t!]
\centering
\centerline{\includegraphics[width=0.85\columnwidth]{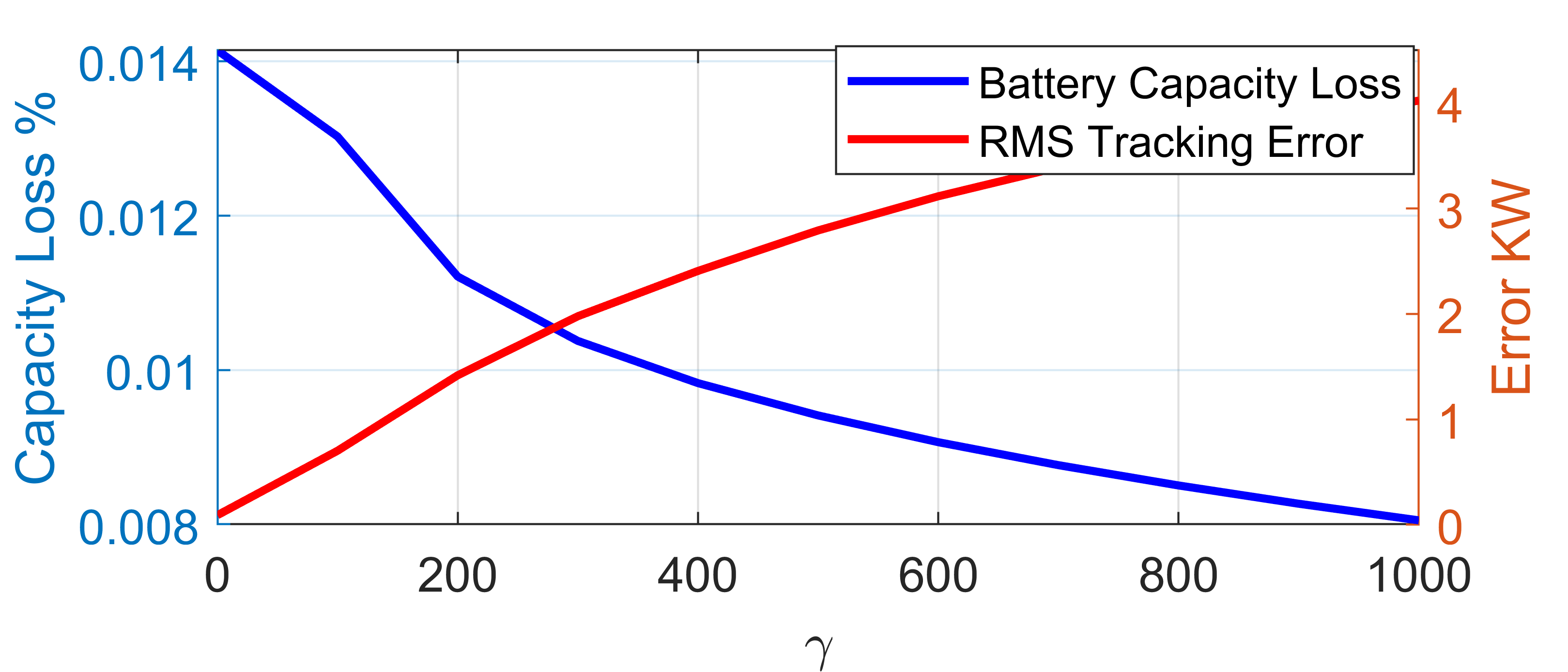}}
\caption{Figure shows the battery capacity loss versus the weight gamma and the power tracking error versus the gamma weight.}
\label{Battery_Capacity}
\end{figure}

\begin{figure}[t!]
\centering
\centerline{\includegraphics[width=0.95\columnwidth]{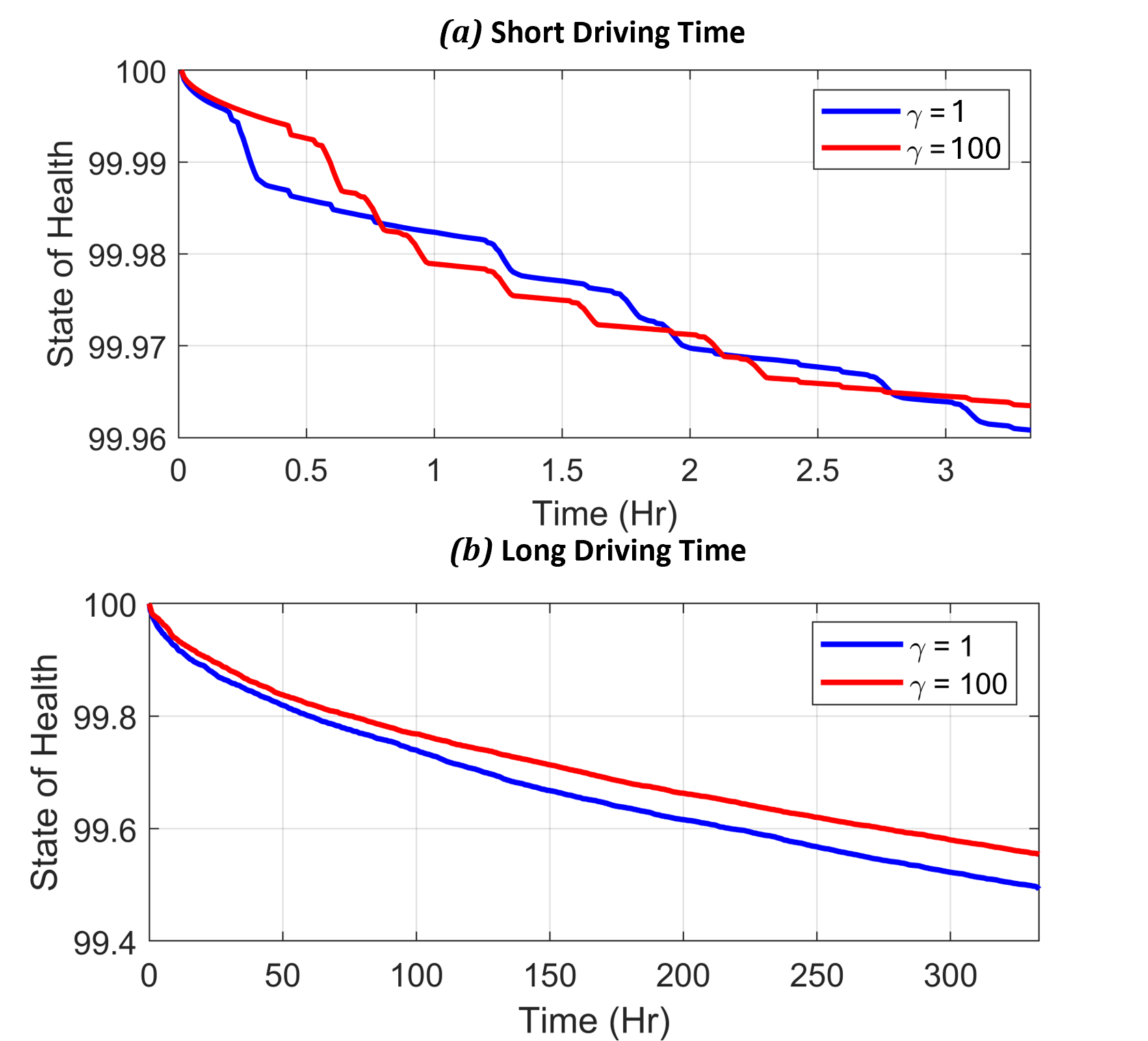}}
\caption{State of Health deterioration for different values of $\gamma$ for a hybrid electric vehicle, (a) for a short driving duration, (b) for a long driving duration.}
\label{SoH_Capacity}
\end{figure}

Fig.~\ref{Power_Track} shows the high-level energy management. The absolute total power demand and the engine and battery combined tracked power in \textsf{kW}, the optimal engine power in \textsf{kW}, the optimal battery power in \textsf{kW}, and the state of charge of the battery are demonstrated for a real data using the US06 (rural), and New York city cycle (NYCC: city) drive-cycle data. Two scenarios are simulated 1) city first, rural next and 2) rural first, city next. It can be seen that the battery is used conservatively and is only deployed when there is a high power demand scenario. Another observation regarding the SoC of the battery and its charging and discharging pattern should also be observed. During the normal power operation, the battery can be seen charging for rural first and city next. Whenever, the demand is not high, the controller charges the battery to be around the SoC of 0.8 (which is the upper limit).

Next, the impact of the battery weight $\gamma$ on the battery operation is studied. The value of $\gamma$ is tuned from $[0-1000]$ while keeping the other weights constant. Fig.~\ref{Battery_Capacity} presents the relationship between $\gamma$ and the battery capacity loss \%. As the value of $\gamma$ increases, the battery capacity loss \% decreases indicating less battery usage. However, it can also be seen from the plot that the power tracking error in \textsf{KW} increases. Thus, the trade-off between the battery health and tracking error is inferred from the results.

Finally, the effect of the weight $\gamma$ is tested using three real drive-cycle data sets (US06, NYCC, SC03). The choice behind this selection is to incorporate both the highway (US06, SC03) and the city driving (NYCC) conditions. The drive cycles were randomized for every simulation run and around 2000 simulations were performed accumulating to the total driving time of around 360 hours. The aforementioned simulation was performed for different weights of $\gamma$ (1 and 100). Fig.~\ref{SoH_Capacity} shows the result of the above simulated scenarios. While, during the short drive time i.e. around 2-3 hours of driving, the effect of $\gamma$ on battery health cannot be seen clearly. But it can be noticed that, as the driving time increases, even with different randomized drive-cycles, the effect of the weight $\gamma$ begins to appear.

\begin{figure}[t!]
\centering
\centerline{\includegraphics[width=0.95\columnwidth]{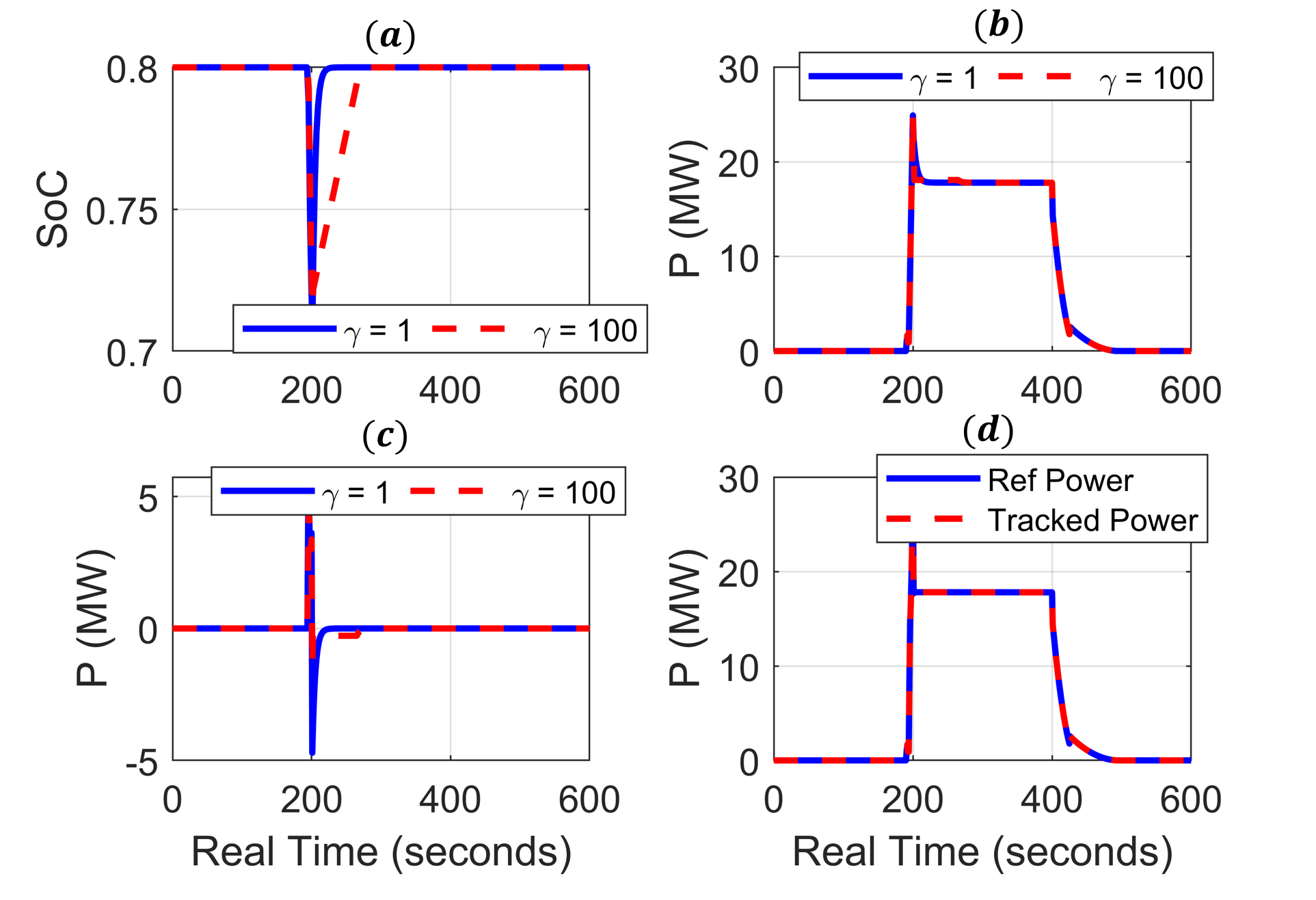}}
\caption{(a) State of charge, (b) engine power, (c) battery power and (d) power tracking for a dynamically positioned ship while anchoring, while ramping to cruising speed and de-ramping back to anchoring for $\gamma=1$ and $\gamma = 100$.}
\label{DPS_Powertrack}
\end{figure}

\begin{figure}[t!]
\centering
\centerline{\includegraphics[width=0.85\columnwidth]{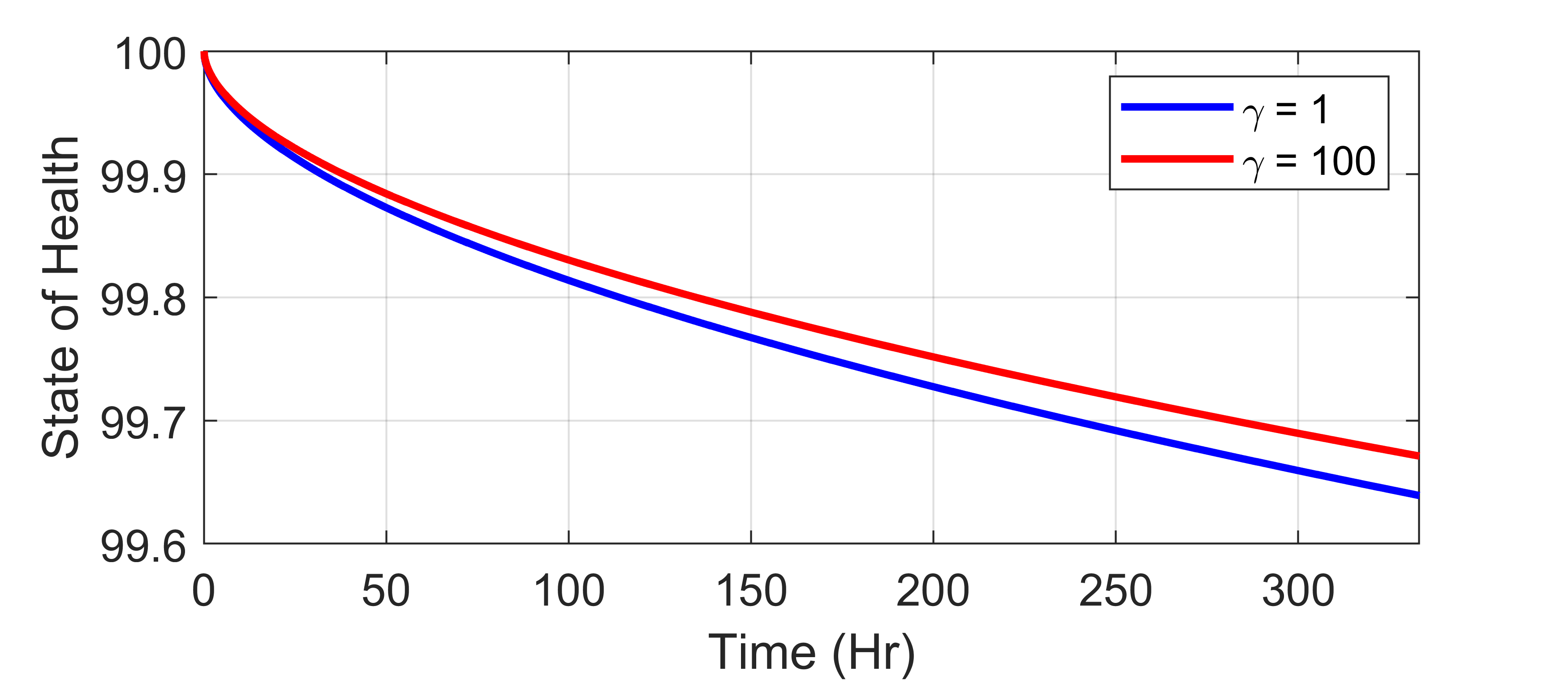}}
\caption{State of Health deterioration for different values of  $\gamma$ for a dynamically positioned ship.}
\label{DPS_SoH}
\end{figure}

\subsection{Dynamically Positioned Hybrid Ships}
Next, we test the developed distributed algorithm to facilitate battery health-conscious power split for a dynamically positioned hybrid ship. To that extent, we design a velocity reference profile in such a way that the ship is initially anchored and then ramps up swiftly to the cruising speed (typically $10$\textsf{m/s} or $19.4$\textsf{knots}) and then de-ramps slowly back to anchoring position. The optimization parameters fixed across the simulations are $\alpha=1000$, $\beta=1$, $\rho=0.1$, and $h=5$\textsf{seconds}. The simulation is run at a fixed time-step of $1$\textsf{ms}, with the optimization problem running every $1$\textsf{second}. The mass of the ship is assumed to be around $1000$\textsf{ton}. The rated engine power is $40$\textsf{MW}, the rated battery power is $10$\textsf{MW}, the state of charge upper and lower limits $0.8$, and $0.5$, ramp-rate of engine 10\% of rated engine power, ramp-rate of the battery 90\% of the rated battery power. Fig.~\ref{DPS_Powertrack} shows the state of charge, engine power, battery power, and the power tracking performance of the designed algorithm.

Fig.~\ref{DPS_SoH} shows the state of health of the battery for different values of the battery penalty weight $\gamma$. It can be seen that as the penalty of the battery is increased, more importance is given to mitigating the battery degradation. 

\begin{figure}[t!]
\centering
\centerline{\includegraphics[width=0.95\columnwidth]{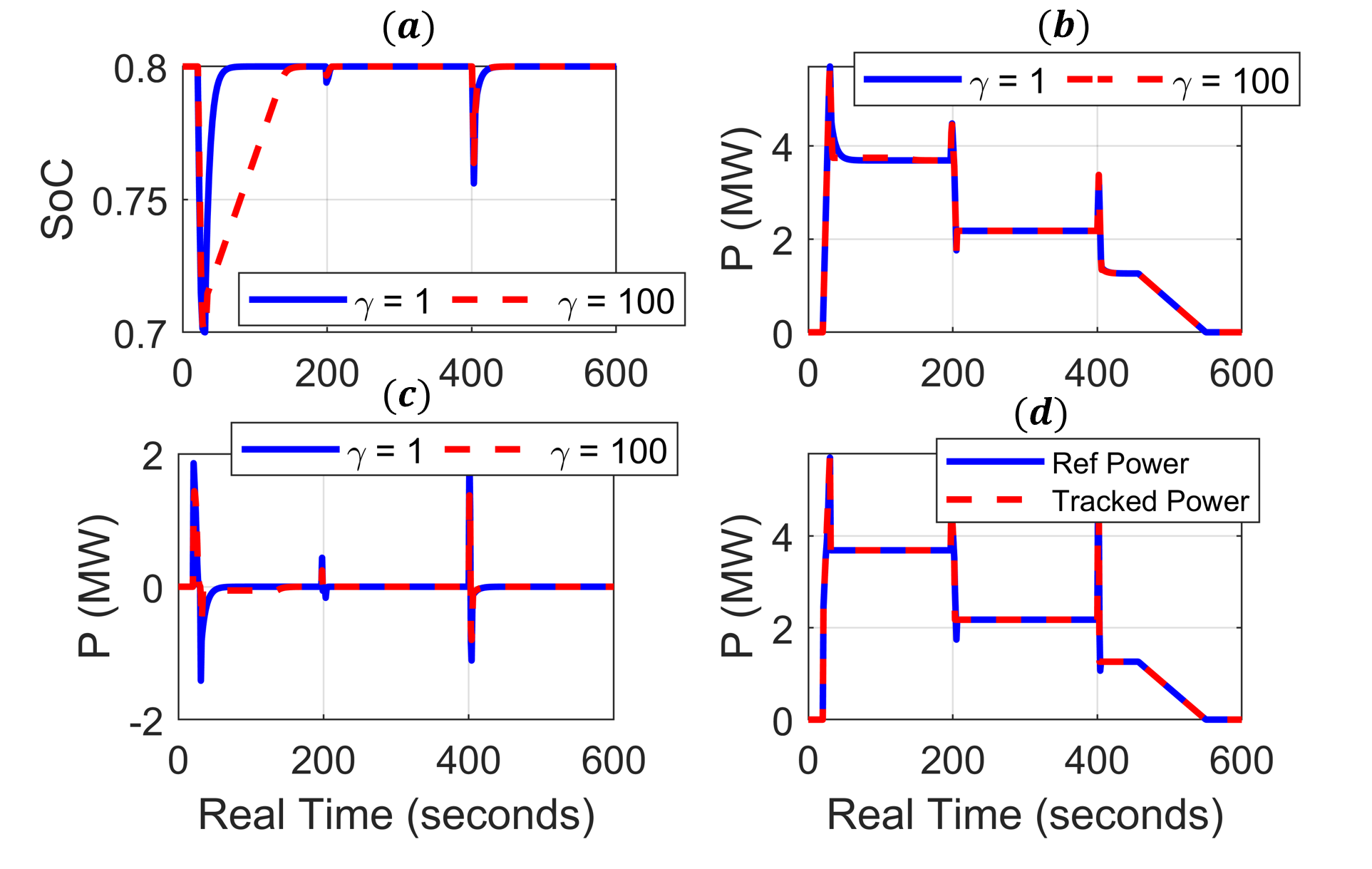}}
\caption{(a) State of charge, (b) engine power, (c) battery power, and (d) power tracking for a hybrid electric aircraft while taking off, cruising at an altitude, and while descending back for $\gamma=1$, and $\gamma = 100$.}
\label{HEA_Powertrack}
\end{figure}

\subsection{Hybrid Electric Aircraft}
Next, we test the developed distributed algorithm on a hybrid electric vehicle, the velocity profile is designed so that the take-off maneuver, followed by fixed altitude cruising action, followed by the descent action are considered. The parameters for the hybrid electric aircraft chosen based on \cite{DOFFSOTTA20206043} as $m_a$ = 42000\textsf{kg}, $\rho_a$ = 1.225$\textsf{Kg-m}^{-2}$, $C_d$ = 0.03, $S$ = 77.3$\textsf{m}^2$. The power ratings considered for the engine and battery are $6$\textsf{MW}, and $2$\textsf{MW}. The ramping capacities are chosen as $10\%$ of rated engine power, and $90\%$ of the rated battery power. Fig.~\ref{HEA_Powertrack} shows the state of charge, engine power, and battery power for multiple values of the battery weights $\gamma = 1$, and $\gamma = 100$. The fixed optimization parameters $\alpha, \beta, h$ are similar to the ones used for HEV and DPS. To avoid redundancy, we do not explicitly present them.

\begin{figure}[t!]
\centering
\centerline{\includegraphics[width=0.85\columnwidth]{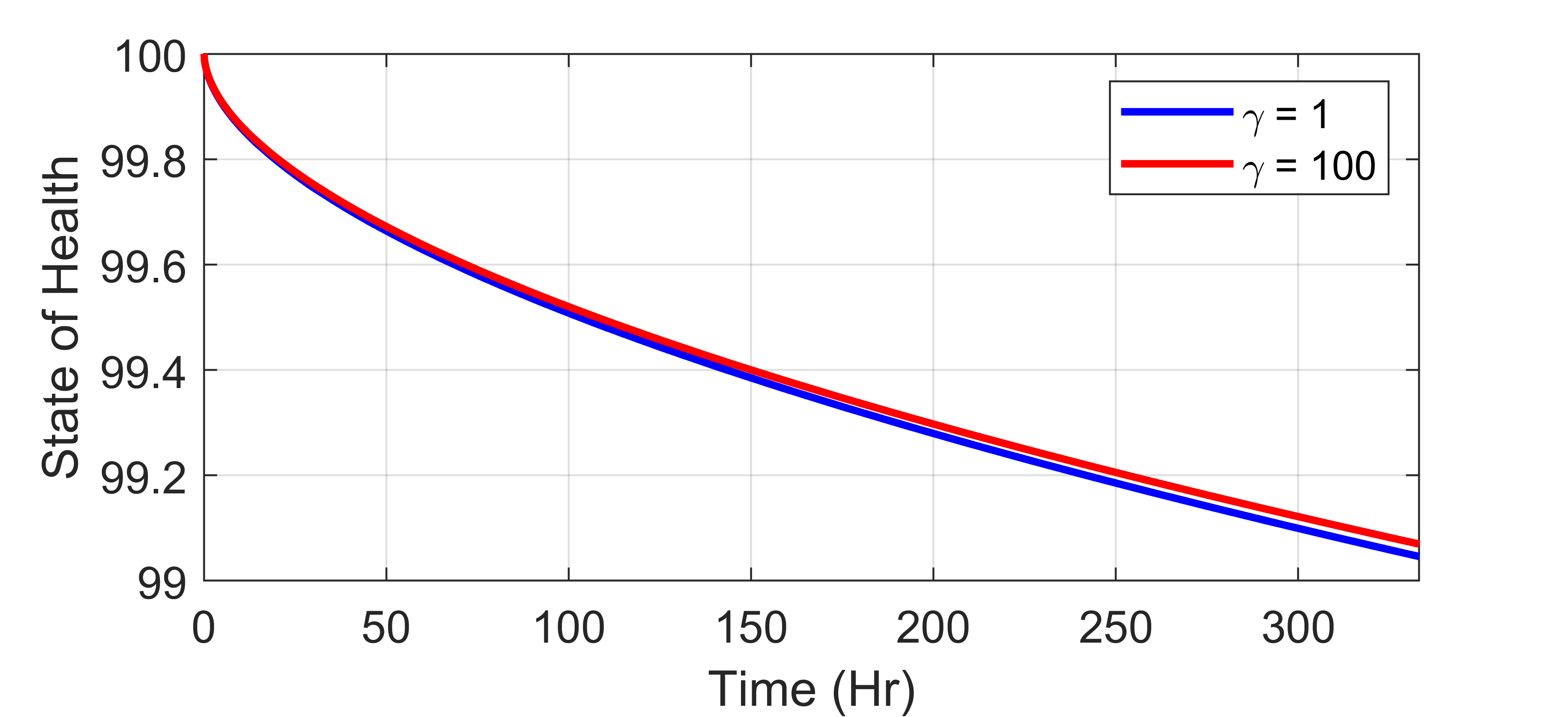}}
\caption{State of Health deterioration for different values of  $\gamma$ for a hybrid electric aircraft.}
\label{HEA_SoH}
\end{figure}

The impact of the different battery penalty weights $\gamma$ is shown in Fig.~\ref{HEA_SoH}. It can be inferred that the power ratings of the engine, and the battery play a key role in determining the battery degradation. Moreover, the HEA power demand also plays an important role in determining battery degradation. 

Overall, across multiple real-time numerical simulations on HEV, DPS, and HEA, the observation remains the same in terms of mitigating battery usage at the expense of more engine usage. The trade-off can be derived between efficient engine operation and conservative battery usage. Thus, the importance of the designed algorithm and the importance of the appropriate battery weighting to mitigate battery degradation has been effectively demonstrated.

\section{Conclusion} \label{Sec: Conclusion}
We considered battery degradation in HEV and developed a corresponding health-heuristic-based model predictive energy management strategy that adaptively splits the power in a distributed framework. The choice of minimizing the battery power as a heuristic to mitigate the battery degradation is shown, through simulation by considering realistic drive cycle data from the U.S. EPA for hybrid electric vehicles and a simulation duration of 360 hours for hybrid electric vehicles, dynamically positioned hybrid ships, and hybrid electric aircraft to improve battery longevity while not sacrificing engine-efficient operation. The results presented demonstrate the operation of the designed algorithm and its effectiveness in mitigating battery degradation. Although the splitting problem was completely derived and results demonstrated via a more realistic simulation, as compared to related works in literature, the overall stability assessment with the designed algorithm relied on well-known stability properties of the MPC problem. The effect of the delays in the algorithm will be reported in the future extensions of this work.

\bibliographystyle{IEEEtran}
\bibliography{myreferences}

\end{document}